\documentclass[11pt]{amsart}

\usepackage{amsthm,amssymb,amsmath,amsopn,amsfonts,etoolbox,mathptmx}
\usepackage{stmaryrd,calc,mathrsfs,enumerate}
\usepackage{tikz}\usetikzlibrary{matrix,arrows} 

\makeatletter
\patchcmd{\@settitle}{\uppercasenonmath\@title}{}{}{}
\patchcmd{\@setauthors}{\MakeUppercase}{\scshape}{}{}
\patchcmd{\section}{\scshape}{\bfseries}{}{}
\renewcommand{\@secnumfont}{\bfseries}
\patchcmd{\abstract}{\scshape\abstractname}{\textbf{\abstractname}}{}{}
\makeatother

\usepackage{hyperref} 
\hypersetup{
	colorlinks=true,
	citecolor=blue,
	linkcolor=red
}

\newtheorem{thm}{Theorem}[section]
\newtheorem{prop}[thm]{Proposition}
\newtheorem{lem}[thm]{Lemma}
\newtheorem{cor}[thm]{Corollary}

\theoremstyle{definition}
\newtheorem{definition}[thm]{Definition}
\newtheorem{example}[thm]{Example}
\newtheorem*{claim}{Claim}

\theoremstyle{remark}
\newtheorem{remark}[thm]{Remark}

\numberwithin{equation}{section}

\newcommand{\bdef}{\begin{definition}}\newcommand{\ndef}{\end{definition}}
\newcommand{\bteo}{\begin{thm}}\newcommand{\nteo}{\end{thm}}
\newcommand{\bprop}{\begin{prop}}\newcommand{\nprop}{\end{prop}}
\newcommand{\brmk}{\begin{remark}}\newcommand{\nrmk}{\end{remark}}
\newcommand{\bcor}{\begin{cor}}\newcommand{\ncor}{\end{cor}}
\newcommand{\blem}{\begin{lem}}\newcommand{\nlem}{\end{lem}}
\newcommand{\bex}{\begin{example}}\newcommand{\nex}{\end{example}}
 \newcommand{\Z}{\mathbb{Z}}

\renewcommand{\P}{\mathbb{P}}
\renewcommand{\O}{\mathcal{O}}

\newcommand{\g}{\mathcal}

\newcommand{\qt}[1]{``#1''}

\DeclareMathOperator{\Pic}{Pic}

\DeclareMathOperator{\gon}{gon}
\DeclareMathOperator{\gengon}{gengon}
\DeclareMathOperator{\Cliff}{Cliff}

\begin{document}


\title{Special divisors on curves on K3 surfaces carrying an Enriques involution}
\author{Marco Ramponi}
\address{Laboratoire de Math\'{e}matiques et Applications, Universit\'{e} de Poitiers, F-86962 Poitiers, France}
\email{marco.ramponi@math.univ-poitiers.fr}


\begin{abstract}
We study the pencils of minimal degree on the smooth curves lying on a K3 surface $X$ which carries a fixed-point free involution. Generically, the gonality of these curves is totally governed by the genus 1 fibrations of $X$.


\end{abstract}

\maketitle

\section{Introduction}
The \emph{gonality} and the \emph{Clifford index} of a smooth algebraic curve $C$ of genus $g\geq2$ are two natural numerical invariants which, roughly speaking, measure the \qt{speciality} of the point $[C]\in\g{M}_g$. (For more details, see Section \ref{gon-and-Cliff} below).

In general, it is very difficult to determine them for a given curve, except in very special cases. Much work has been done regarding the behaviour of these invariants for curves lying on some special classes of surfaces. For example, for a curve $C$ lying on a smooth projective surface, an interesting question is to ask whether its gonality, or the Clifford index, varies when $C$ moves in its linear system. A classical result in this direction, essential in this work, is the theorem of Green and Lazarsfeld \cite{GL}, which states that, when $C$ lies on a K3 surface, 
\begin{enumerate}[(i)]
\item The Clifford index is constant among smooth curves in the linear system $|C|$.
\item If $\Cliff(C)<\lfloor\frac{g-1}{2}\rfloor$, then there exists a divisor $D$ on the K3 surface such that $A=\O_C(D)$ appears in the definition of $\Cliff(C)$ and attains the minimum.
\end{enumerate}

The possibility of an explicit computation of the Clifford index and the gonality of curves lying on particular K3 surfaces is an interesting and challenging question. This highly depends on the geometry of the given ambient surface. We are thus interested to study this problem for some special classes of K3 surfaces carrying rich structure, such as automorphisms. At least, we want to consider surfaces with Picard number $\rho\geq2$, for if the Clifford index of a curve on a K3 surface $X$ with cyclic Picard group is cut out by a divisor $D$, then a simple computation shows that $D$, up to linear equivalence, coincides with the ample generator of $\Pic(X)$. This situation is not very interesting.

In this note, we consider a pair $(X,\vartheta)$ where $X$ is a K3 surface and $\vartheta$ is a fixed-point-free involution. It is well-known \cite[VIII.19]{BHPV} that there is a $10$-dimensional moduli space of such pairs and the generic pair is such that there are no smooth rational curves lying on $X$. Our main result is the following.

\bteo\label{main} Let $(X,\vartheta)$ be a generic K3 surface $X$ with an Enriques involution $\vartheta$. The gonality of any smooth curve $C$ on $X$, with $C^2>0$, is cut out by any elliptic curve $E$ on $X$ having minimal intersection with $C$, i.e. $\gon(C) = E\cdot C$.
\nteo

\brmk If $X$ is as in Theorem \ref{main}, by the results of \cite{Kn09} and the fact that there are no classes of square $\pm 2$ on $X$ (see (\ref{lattice}) below), any smooth curve on $X$ is such that $\Cliff(C)=\gon(C)-2$. Therefore, Theorem \ref{main} automatically contains a statement about the Clifford index of $C$, and the elliptic curve $E$ has the role of $D$ in part (ii) of the result of Green and Lazarsfeld stated above.
\nrmk

\brmk We point out that a result of a similar flavor is given in the pioneer work of Reid \cite{Reid}. Consider a complete base-point free pencil of degree $d$ on a curve $C$ lying on a K3 surface. The inequality
$$\frac{1}{4}d^2 + d + 2< g(C)$$
is then a sufficient condition for the pencil to be induced by an elliptic pencil on the surface. Of course, even when $d=\gon(C)$, this condition fails for many curves on the K3 surfaces considered in Theorem \ref{main} -- e.g., whenever $C$ is the pull-back of a curve on the quotient Enriques surface having maximal gonality (these curves always exist, as we now recall).
\nrmk

As well as \cite{GL}, the second fundamental result for the proof of Theorem \ref{main} is the work of Knutsen and Lopez \cite{KL1} where the authors carry out a detailed analysis of the gonality of curves on an Enriques surface. We have \cite[Theorem~1.3]{KL1}, for a curve $C$ with $C^2>0$ on an Enriques surface $Y$,
\begin{align}\label{gengon}
\gengon|C| = \min\{ 2\phi(C), \mu(C), \lfloor\frac{C^2}{4}\rfloor + 2\},
\end{align}
where
\begin{align*}
\phi(C) &= \min\{\lvert F\cdot C\rvert \, \colon \, F\in\Pic(Y), F^2=0, F\not\equiv0 \}, \\
\mu(C) &= \min\{B\cdot C -2 \, \colon \, F\in\Pic(Y), B>0, B^2=4, B\not\equiv C \}
\end{align*}
and $\gengon|C|$ denotes the gonality of the \emph{general} curve in $|C|$. Indeed, it may well happen that there exist linear subsystems of $|C|$ whose smooth curves have lower gonality than $\gengon|C|$ (cf.~\cite[Corollary 1.6]{KL1}). 

In light of this and of the trichotomy expressed by (\ref{gengon}), the content of Theorem \ref{main} might be somewhat surprising at a first glance, because the K3 surface $X$ contains in particular the curves which are  pulled back from the Enriques surface $Y=X/\langle\vartheta\rangle$, for which (\ref{gengon}) holds. The point is that, while the gonality of a curve on an Enriques surface might well be lower than the minimal degree induced by genus 1 pencils of the surface, we still have the following condition:
\begin{align}\label{inequality-phi}
2\phi(C)\leq \gengon|C| +2
\end{align}
(cf.\cite[Corollary 1.5]{KL1}). In section \ref{section-proof} below, we prove Theorem \ref{main} by showing how the condition (\ref{inequality-phi}) is essentially enough to deduce that the gonality of all curves on the relative K3 cover $X$ is governed by the elliptic pencils pulled back on $X$ from its Enriques quotient.

Throughout the paper we work over the field of complex numbers. 

\subsection*{Acknowledgements}
Thanks to Alessandra Sarti whose remarks and suggestions have considerably improved the exposition of this paper. I wish to thank Andreas Leopold Knutsen for useful conversations under the shade of an oak in the university park of Bergen. Thanks to an anonymous referee for a suggestion which simplified and shortened an argument in the proof.

\section{Preliminaries on gonality and Clifford index}\label{gon-and-Cliff}

For all basic results and implications coming from Brill-Noether theory, in this section we refer the reader to \cite{ACGH}.

Let $C$ be a smooth algebraic curve of genus $g\geq2$. The \emph{gonality} of $C$ is defined as the integer
$$\gon(C) = \min\{ \deg(A) \, \colon \, A\in\Pic(C), \, h^0(A)=2\}. $$
In particular, $\gon(C)=2$ if and only if $C$ is hyperelliptic. By Brill-Noether theory,
\begin{align}\label{maximal-gonality}
\gon(C)\leq \lfloor \frac{g+3}{2} \rfloor,
\end{align}
with an equality for the general curve in $\g{M}_g$. We refer to $\lfloor\tfrac{g+3}{2}\rfloor$ as the \emph{maximal gonality} of a curve of genus $g$.

For a line bundle $A$ on $C$, one defines $\Cliff(A)=\deg A-2h^0(A)+2$. The \emph{Clifford index} of the curve $C$ itself is then defined as
$$\Cliff(C) = \min\{\Cliff(A) \, \colon \, A\in\Pic(C), \, h^0(A)\geq2, \, h^1(A)\geq2\}. $$

Note that, by Brill-Noether theory, the line bundles appearing in the definition of $\Cliff(C)$ always exist for $g\geq4$. Thus, when $g=2,3$ we adopt the standard convention that $\Cliff(C)=0$ when $C$ is hyperelliptic and $\Cliff(C)=1$ otherwise.

By Clifford's theorem, we have $\Cliff(C)\geq0$ with equality if and only if $C$ is hyperelliptic. Since $\Cliff(C)\leq\gon(C)-2$, we have
$$\Cliff(C)\leq \lfloor \frac{g-1}{2} \rfloor,$$
and the equality holds for the general curve in $\g{M}_g$. 

Gonality and Clifford index are indeed very much related: for any curve $C$ of Clifford index $c$ and gonality $k$, one has \cite{CM91}
\[c+2\leq k \leq c+3,\]
and curves for which $k=c+3$ are conjectured to be very rare \cite{ELMS}.

As was recalled in the Introduction, if $C$ lies on a K3 surface $X$ and has non-maximal Clifford index, then by \cite{GL}, there exists a line bundle $M$ on the surface such that $\Cliff(C)=\Cliff(M\otimes\O_C)$. By \cite[Lemma 8.3]{Kn01}, building on \cite{Ma}, one can choose $M$ such that $h^0(M\otimes\O_C)=h^0(M)$ and $h^1(M)=0$, whence
\begin{align}\label{cliff-explicit}
\Cliff(C)=C\cdot M - M^2-2.
\end{align}
Moreover,
\begin{align}\label{M-irreducible}
M \mbox{ is represented by an irreducible curve,}
\end{align}
and
\begin{align}\label{M-inequality-Knutsen}
2M^2\leq M\cdot C, \, \mbox{ with equality only if } C\sim 2M.
\end{align}

\section{Proof of Theorem \ref{main}}\label{section-proof}
Let $(X,\vartheta)$ be a pair consisting of a K3 surface $X$ together with a fixed-point free \emph{involution} (i.e. an order 2 automorphism) $\vartheta$ of $X$. We denote by  
$$Y=X/\langle\vartheta\rangle$$
the quotient surface, which we call an \emph{Enriques surface}. 

We let $\pi\colon X\to Y$ denote the natural projection and by 
$$\pi^\ast\colon H^2(Y,\Z)\to H^2(X,\Z); \qquad \pi_\ast\colon H^2(X,\Z)\to H^2(Y,\Z)$$
the natural induced maps, satisfying
$$\pi_\ast\pi^\ast(y)=2y; \quad \pi^\ast\pi_\ast(x)=x+\vartheta^\ast(x); \quad (\pi^\ast y_1,\pi^\ast y_2) = 2(y_1,y_2). $$

If we let $H^2(X,\Z)^\vartheta$ be the set of classes in $H^2(X,\Z)$ which are fixed by $\vartheta$, the above properties easily imply that the restriction of $\pi_\ast$ to $H^2(X,\Z)^\vartheta$ is an isomorphism onto its image which multiplies the intersection form by $2$. That is,
\begin{align}\label{intersection-times-two}
\pi_\ast(H^2(X,\Z)^\vartheta)\simeq H^2(X,\Z)^\vartheta(2)
\end{align}

We recall that $\vartheta$ is a \emph{non-symplectic} involution, in the sense that it acts by multiplication by $-1$ on $H^{2,0}(X)$. This, together with the fact that the action of $\vartheta$ on $H^2(X,\Z)$ preserves the intersection form, yields
$$H^2(X,\Z)^{\vartheta} \subset H^2(X,\Z) \cap H^{1,1}(X).$$

By the Lefschetz theorem on $(1,1)$-classes, we identify the member on the right hand side of the above equation with $\Pic(X)$, the Picard group of $X$. 

As shown in \cite{DK}, one can construct a 10-dimensional period domain $\g{D}$ for the pairs $(X,\vartheta)$, and for the generic such pair in $\g{D}$ one has the equality
\begin{align}\label{generic-assumption}
H^2(X,\Z)^\vartheta = \Pic(X).
\end{align}

From now on, we assume $(X,\vartheta)$ to satisfy condition (\ref{generic-assumption}) above. This is our genericity assumption in the statement of Theorem \ref{main}.

Let $L$ be a big and nef line bundle on $X$. We deduce two immediate consequences of (\ref{generic-assumption}); one purely numerical and a second one more geometric in nature. 

Firstly,  (\ref{generic-assumption}), together with (\ref{intersection-times-two}), yield
\begin{align}\label{lattice}
L^2 \equiv 0 \mod 4.
\end{align}
In particular, $X$ contains no classes of self-intersection $\pm 2$. By \cite{Kn09} this implies that the gonality of smooth curves $\Sigma$ in $|L|$ is constant and 
$$\Cliff(\Sigma)=\gon(\Sigma)-2.$$ 
Thus, whenever $\Cliff(\Sigma)$ is computed by the restriction of a divisor $D$ on the surface, since $\deg_\Sigma(D)=D\cdot\Sigma\in2\Z$ (again by (\ref{intersection-times-two})), we see that both the Clifford index and the gonality of $\Sigma$ must be even.

Secondly, (\ref{generic-assumption}) implies that $\vartheta$ acts as an involution on $|L|\simeq\P^g$. This lifts to an involution 
$$\vartheta^\ast\colon H^0(X,L)\to H^0(X,L),$$ 
at the level of sections. Let us denote by $V_\pm\subset H^0(X,L)$ the eigenspaces relative to the eigenvalues $\pm 1$ for this action. The sections in $V_+$ and $V_-$ yield the effective divisors in $|L|$ which are mapped to themselves by $\vartheta$. With respect to the covering $\pi\colon X\to Y$, these divisors map $2$ to $1$ onto divisors on the Enriques quotient. In other words, we may choose an effective divisor $C\subset Y$ such that $\pi^\ast C$ belongs to $|L|$ and the linear subspace $\P_+=\P(V_+)$, as a subsystem of $|L|$, corresponds to $\pi^\ast|C|$. (With respect to this choice, $\P_-=\P(V_-)$ corresponds then to $\pi^\ast|C+K_Y|$). 

As $L^2>0$ by assumption, we have $C^2>0$, hence the general member of $|C|$ is a smooth irreducible curve. In fact, if $|C|$ is hyperelliptic then its general member is a smooth (hyperelliptic) curve by \cite[Corollary 4.5.1 p.\,248]{CD1}. Else, $|C|$ is basepoint free \cite[Proposition 4.5.1]{CD1} and we apply Bertini's theorem. We therefore assume $C$ itself to be a smooth irreducible curve. 

Moreover, we choose $C$ to be general in its linear system, so that, following \cite{KL1}, the gonality of $C$ is equal to the \emph{general gonality} (i.e.\,the greatest gonality among the smooth curves in $|C|$)
$$\gon(C)=\gengon|C|.$$

Let $\widetilde{C}:=\pi^\ast C$. It is well-known that the restriction of the canonical bundle $K_Y$ to $C$ is non-trivial. It follows that $\widetilde{C}$ is a smooth irreducible curve of genus $g$ in $|L|$ and the restriction of the covering map
$$\pi|_{\widetilde{C}}\colon \widetilde{C} \to C$$ 
exhibits $\widetilde{C}$ as an unramified double covering of $C$. In particular, by push-forward of a pencil of minimal degree on $\widetilde{C}$, or by pull-back of gonality pencils from $C$,
\begin{align}\label{inequality-double-cover}
\gon(C)\leq\gon(\widetilde{C})\leq 2\gon(C).
\end{align}

Let now $|2F|$ be a genus 1 pencil on the Enriques surface $Y$ such that 
$$\phi(C)=F\cdot C.$$
We set $\widetilde{F}=\pi^\ast F$ and by (\ref{inequality-phi}) we obtain the following inequality
\begin{align}\label{inequality-gon}
\gon(\widetilde{C}) \leq \widetilde{F}\cdot\widetilde{C}  = 2\phi(C) \leq \gon(C)+2
\end{align}

We claim that the first inequality is, in fact, always an equality. 

Indeed, assume by contradiction $\gon(\widetilde{C})< \widetilde{F}\cdot\widetilde{C}$. By (\ref{inequality-double-cover}),
\[\gon(C)<2\phi(C).\]

Applying \cite[Corollary 1.5]{KL1}, we have
\begin{align}\label{possible-cases}
C^2\geq 10 \ \mbox{ or }  (C^2,\phi(C))=(6,2) \ \mbox{ or } (C^2,\phi(C))=(4,2).
\end{align}

\begin{claim}
$\gon(\widetilde{C}) < \lfloor\frac{g(\widetilde{C})+3}{2}\rfloor$ (in particular, $\gon(\widetilde{C})$ is even).
\end{claim}

\begin{proof}
If $C^2=4$, then $\widetilde{C}^2=8$ and $g(\widetilde{C})=5$, so if equality holds, then it must be $\gon(\widetilde{C})=4=2\phi(C)=\widetilde{F}\cdot \widetilde{C}$, a contradiction.

If $C^2\geq6$, then $\widetilde{C}^2\geq12$, so that $g(\widetilde{C})\geq7$. Hence
\begin{align*}
\gon(\widetilde{C}) & \leq \widetilde{F}\cdot \widetilde{C} -1 \\
& \leq \gon(C)+1 \\
& \leq \lfloor\frac{g(C)+3}{2}\rfloor +1 \\
&  = \lfloor\frac{\frac{g(\widetilde{C})+1}{2}+3}{2}\rfloor +1 \\
& = \lfloor\frac{g(\widetilde{C})+11}{4}\rfloor \\
& < \lfloor\frac{g(\widetilde{C})+3}{2}\rfloor
\end{align*}
where the last inequality uses $g(\widetilde{C})\geq7$.
\end{proof}

Since, by our assumptions, $\gon(C)<2\phi(C)$, we proceed as follows.

If $\gon(C)=2\phi(C)-1$, then $\gon(C)\leq\gon(\widetilde{C})<2\phi(C)$ is incompatible with the parity of $\gon(\widetilde{C})$, whence yielding a contradiction.

By (\ref{inequality-phi}), we may therefore assume 
$$\gon(C)=2\phi(C)-2.$$ 
Then, necessarily $\gon(C)=\gon(\widetilde{C})$. 

We pick a line bundle $M$ on the K3 surface $X$, as in (\ref{cliff-explicit}), i.e.\,such that
$$\Cliff(\widetilde{C})=\Cliff(M\otimes\O_{\widetilde{C}})=M\cdot \widetilde{C} - M^2 -2.$$

If $M^2=0$, then it follows by (\ref{M-irreducible}) that $M$ is represented by an elliptic curve $E$. By construction, the elliptic curve $\widetilde{F}$ has minimal intersection with $\widetilde{C}$ among all elliptic curves on $X$, whence $\gon(\widetilde{C})=\Cliff(\widetilde{C})-2=E\cdot\widetilde{C}=\widetilde{F}\cdot \widetilde{C}$, a contradiction.

We may therefore assume $M^2>0$. Then $M^2\geq4$ by (\ref{lattice}). By (\ref{possible-cases}) and (\ref{cliff-explicit})
\[ 4\leq M^2 \leq M\cdot \widetilde{C} - M^2 = \Cliff(\widetilde{C}) +2 = \gon(\widetilde{C}). \]

Assume $\gon(\widetilde{C})=4$. Then $M^2=4$ and $M\cdot\widetilde{C} =8$, so that (\ref{M-inequality-Knutsen}) gives $\widetilde{C}\sim2M$, whence $\widetilde{C}^2=16$. It follows that $C^2=8$. This contradicts (\ref{possible-cases}) and we may therefore assume 
\[ \gon(\widetilde{C})>4. \]

Arguing as above, $\vartheta$ acts as an involution on $|M|$, and we get subsystems $\P_\pm$ of $|M|$, corresponding to $\pi^\ast|D|$ and $\pi^\ast|K_Y+D|$, where $D$ is some effective divisor on $Y$, with $\pi^\ast D\sim M$. Since $M^2>0$, we have $D^2>0$, whence
$$h^0(D)\geq2.$$

We have $\pi^\ast(C-D) \sim \widetilde{C}-M$, whence $(C-D)^2>0$. Also, 
$$2(C-D)\cdot C=\pi^\ast(C-D)\cdot\pi^\ast C=N\cdot\widetilde{C} = M\cdot N+N^2 >0,$$ 
so that by Riemann-Roch, $h^0(C-D)\geq2$ and, similarly, $h^0(C-D+K_Y)\geq2$. Therefore, $C\cdot(C-D)\geq2$ by the Hodge index theorem, so that
\begin{align*}
C^2 & = (D + C - D)^2 \\
& = D^2 +(C-D)^2 +2C\cdot (C -D) \\
& \geq 2 +2 +2 \\
& =6
\end{align*}

We may now apply \cite[Lemma 2.3]{KL2} and obtain
\[ \Cliff(C) \leq D\cdot C -D^2. \]

By \cite[Theorem 1.1]{KL2}, unless $C$ is a smooth plane quintic (which has gonality $4$), one has $\Cliff(C)=\gon(C)-2$, and so the above inequality yields
\begin{align}\label{estimate1}
2(D\cdot C-D^2) \geq 2\gon(C)-4.
\end{align}

On the other hand, $\gon(\widetilde{C})= \Cliff(\widetilde{C})+2 =\widetilde{C}\cdot M-M^2$, thus
\begin{align}\label{estimate2}
2(D\cdot C - D^2) =\widetilde{C}\cdot M-M^2=\gon(\widetilde{C}).
\end{align}
Since $\gon(C)=\gon(\widetilde{C})>4$, the equations (\ref{estimate1}) and (\ref{estimate2}) are incompatible.  Hence, our assumption that $\gon(\widetilde{C})<\widetilde{F}\cdot \widetilde{C}$ has led to a contradiction and we conclude
$$\gon(\widetilde{C})=\widetilde{F}\cdot\widetilde{C}.$$

As we have already observed above, thanks to our genericity assumption on $X$, this holds true for all smooth curves $\Sigma$ in $|\widetilde{C}|=|L|$. This concludes the proof of Theorem \ref{main}, \ q.e.d.

\bibliography{bigbib}
\bibliographystyle{plain}

\end{document}